\documentclass[11pt]{amsart} 
\usepackage{graphicx}
\usepackage{rlepsf,epstopdf}
\usepackage[all]{xy}
\usepackage{hyperref}
\usepackage[all]{xy}
\usepackage[usenames]{color}

\newcommand{\rr}{\ensuremath{\mathbb{R}}}
\newcommand{\zz}{\ensuremath{\mathbb{Z}}}
\newcommand{\qq}{\ensuremath{\mathbb{Q}}}

\theoremstyle{plain}
\newtheorem{thm}{Theorem}[section]

\newtheorem{lem}[thm]{Lemma}

\newtheorem{prop}[thm]{Proposition}

\theoremstyle{definition}
\newtheorem{defn}[thm]{Definition}

\theoremstyle{remark}
\newtheorem{rem}[thm]{Remark}
\newtheorem{exam}[thm]{Example}

\numberwithin{equation}{section}

\theoremstyle{plain}
\newtheorem*{thm1}{Theorem~\ref{thm:label}}

\def\dfn#1{{\textbf {#1}}}

\title{Rational Seifert Surfaces in Seifert Fibered Spaces}
\author[J. Licata]{Joan E. Licata}  \address{Institute for Advanced Study,
  Princeton, NJ 08540 \& Australian National University} \email{jelicata@math.stanford.edu} 

\author[J. Sabloff]{Joshua M. Sabloff}\address{Haverford College,
Haverford, PA 19041} \email{jsabloff@haverford.edu} \thanks{JMS is
partially supported by NSF grant DMS-0909273.}

\date{\today}

\begin{document}

\begin{abstract}
Rationally null-homologous links in Seifert fibered spaces may be represented combinatorially via labeled diagrams.  We introduce an additional condition  on a labeled link diagram and prove that it is equivalent to the existence of a rational Seifert surface for the link.  In the case when this condition is satisfied, we generalize Seifert's algorithm to explicitly construct a rational Seifert surface for any rationally null-homologous knot.  As an application of the techniques developed in the paper, we derive closed formulae for the rational Thurston-Bennequin and rotation numbers of a Legendrian knot in a contact Seifert fibered space. 
\end{abstract}

\maketitle
\section{Introduction}
This paper studies rationally null-homologous links in Seifert fibered spaces, with the goal of extending techniques from classical knot theory to a more general setting.   Previous work in this vein includes Gilmer's signatures for rationally null-homologous links \cite{gilmer:ratl-link-cob} and Calegari and Gordon's classification of knots with small rational genus \cite{cg:ratl-genus}. More generally, recent work on the Berge Conjecture has shown that the study of rationally null-homologous links is important for understanding Dehn surgery questions; see, for example, \cite{ras:berge}.  Rationally null-homologous knots are also interesting in a contact geometric setting.  For example, Baker and Etnyre generalized the definition of classical invariants for Legendrian knots to the case of rational homology three-spheres and classified rational Legendrian unknots \cite{be:rational-tb}, and Cornwell has studied Bennequin-type inequalities in lens spaces \cite{cornwell:lens-invts}.  Our interest in this topic was also prompted by contact geometry \cite{joan-josh:lch4sfs}, but we hope the techniques developed in this paper will find applications within the wider context of the link theory in rational homology three-spheres. 

Just as a knot in $\mathbb{R}^3$ is often studied via the combinatorics of its planar projection, we consider the projection of a knot in a Seifert fibered space to its two-dimensional orbifold base. As we show in Section~\ref{sect:sfs}, labeling this projection with some ancillary data permits the  topological type of the knot to be recovered.  Turaev initiated this ``shadow'' approach in the case of knots in an $S^1$ bundle over a surface, and the extension to $S^1$ bundles over orbifolds answers a question he posed in \cite{turaev:shadow}.  

After discussing labeled knot diagrams, we will introduce two further combinatorial objects: a \dfn{formal rational Seifert surface} is an assignment of an integer to each complementary components of the labeled knot projection, while a compatible \dfn{fiber distribution} is an assignment of integers to each quadrant around each double point of a labeled diagram. The precise definitions are given in Section~\ref{sec:frss} and allow us to state the following theorems:

\begin{thm}\label{thm:label} If $K$ is rationally null-homologous in a Seifert fibered space, then any labeled diagram for $K$ admits a formal rational Seifert surface with a compatible fiber distribution.
\end{thm}

\begin{thm}\label{thm:bounds} If a labeled diagram for $K$ admits a formal rational Seifert surface with a compatible fiber distribution, then $K$ bounds a rational Seifert surface in $M$.
\end{thm}

It is clear that $K$ bounding a rational Seifert surface implies that $K$ is rationally null-homologous; thus, these two theorems also show that the existence of a formal rational Seifert surface with a compatible fiber distribution is equivalent to the geometric condition that $K$ is rationally null-homologous.

A key construction in this paper is a generalization of Seifert's algorithm for knots in $\rr^3$ to rationally null-homologous knots in Seifert fibered spaces.  This algorithm, which provides the proof of Theorem~\ref{thm:bounds}, explicitly constructs a rational Seifert surface in $M$ from the given combinatorial data. The algorithm is described in Section~\ref{sec:seifert-algorithm}.

In the final section, we turn our attention to the special case of a Legendrian knot in a Seifert fibered space equipped with a transverse, $S^1$-invariant contact structure.  (This setting was studied in more detail in \cite{joan-josh:lch4sfs}.)  As an application of the algorithm defined in Section~\ref{sec:seifert-algorithm}, we compute the rational classical invariants of a Legendrian knot from its labeled diagram; this result generalizes the familiar formulae for classical invariants in the standard contact $\mathbb{R}^3$.

\begin{prop} \label{prop:initial-leg} Let $K$ be a rationally null-homologous Legendrian knot in a contact Seifert fibered space.  The rational rotation number of $K$ may be computed directly from a formal rational Seifert surface, and the rational Thurston-Bennequin number of $K$ may be computed directly from a compatible fiber distribution.\end{prop}

See Proposition~\ref{prop:legendrian} for a more precise statement.

% **********
\section{Labeled diagrams}

\subsection{Background}\label{sect:sfs}

We view Seifert fibered spaces as $S^1$ bundles over two-dimensional orbifolds, following the notational conventions of \cite{lisca-matic:transverse, massot}.  
 
Let $\Sigma'$ be an oriented surface, possibly with boundary, with $r+1$ discs removed from its interior.  Orient the new components of $\partial \Sigma'$ as the boundary of the missing disc, and let $M' =\Sigma' \times S^1$.  The first homology groups of the boundary tori of $M'$ are generated by classes $\langle m_i, \ell_i \rangle$, with  $\cup_i m_i=[\partial \Sigma'\times \{\text{pt}\}]$ and $\ell_i = [\{\text{pt}\} \times S^1]$, oriented so that $m_i \cdot \ell_i = 1$. Note that this orients all the fibers in $M$.

For $1\leq i\leq r$, let  $\alpha_i$ and $\beta_i$ be relatively prime integers satisfying $0 < \beta_i <\alpha_i$. Glue a solid torus $W_i$ to the $i^{th}$ boundary component of $M'$ so that the image of a meridian represents the homology class $\alpha_i m_i+ \beta_i \ell_i$. To the remaining boundary component, glue a solid torus  so that the meridian is sent to a curve representing the class of  $m_0 + b \ell_0$.  The fiber structure on the boundary of $M'$ extends uniquely to a fiber structure on the interior of the surgery solid tori, and the resulting identification space $M$ is said to have  \dfn{Seifert invariants} $(g,b; (\alpha_1, \beta_1), \ldots, (\alpha_r, \beta_r))$. Note if $\Sigma'$ has boundary, then $M$ has an $S^1$-fibered boundary.

Every orientable Seifert fibered space with an orientable fiber space can be realized via this construction; given two Seifert invariants, it is straightforward to determine whether they correspond to the same Seifert fibered manifold \cite{orlik}. The \dfn{rational Euler number} of a Seifert fibered space with Seifert invariants $(g,b; (\alpha_1, \beta_1), \ldots, (\alpha_r, \beta_r))$ is the rational number
\[e(M)=-b-\sum_{i=1}^r \frac{\beta_i}{\alpha_i} .\]

% *****
\subsection{Labeled Diagrams}
\label{sec:labeled-diagrams}

Let $L$ be an oriented link in a Seifert fibered space $M$ with Seifert invariants $(g,b; (\alpha_1, \beta_1), \ldots, (\alpha_r, \beta_r))$.  We suppose throughout that $L$ is everywhere transverse to the fibers, and we let $(\Sigma, \Gamma_L)$ denote the image of $(M, L)$ under the quotient map $\pi$ which sends each fiber to a point.  In order to recover the isotopy class of $L$ from this projection, we will use a \dfn{labeled diagram}; this notion was introduced in \cite{s1bundles} and is similar to Turaev's notion of a shadow for a link in a circle bundle \cite{turaev:shadow}.  

The fiber over a double point of $\Gamma_L$ is separated by its intersections with $L$ into two oriented chords, and we systematically select a preferred chord at each crossing.     Near a crossing, there is a unique quadrant which is coherently and positively oriented by $L$.  Declare this quadrant and the opposite quadrant to be \dfn{positive}, and declare the adjacent quadrants to be \dfn{negative}.  When the oriented boundary of a positive (respectively, negative) quadrant is lifted to segments of $K$ connected by a chord, the preferred chord is the one traversed positively (negatively).

Given a region $R$ in $\Sigma\setminus \Gamma_L $, define $M_R$ to be the restriction of the orbibundle $M \to \Sigma$ to $R$.  Let $A_R$ be the least common multiple of the orders of the orbifold points in $R$; if $R$ contains no orbifold points, set $A_R = 1$.  The subcurves of $L$ which project to $\partial R$ may be concatenated with the preferred chords over the corners of $R$ to yield a closed curve $L_R$ in $\partial M_R$; orient $L_R$ so that the orientation induced by its projection to $\Sigma$ agrees with that of $\partial R$.  Let $\widetilde{R}$ be the the $A_R$-fold branched covering $\widetilde{R}$ of $R$. Use the covering map to pull back the bundle $M_R$ to $\widetilde{R}$.  This lifts $L_R$ to a closed $1$-manifold $\widetilde{L}_R$ in an honest $S^1$ bundle over $\widetilde{R}$.

Let 
\[\gamma_1 \times \cdots \times \gamma_{k_R}: S^1 \sqcup \cdots \sqcup S^1 \rightarrow S^1\times \partial \widetilde{R}\] denote the map whose image is $\widetilde{L}_R$.  Choose a trivialization of $S^1\times \widetilde{R}$ and let $\iota:S^1\times \partial \widetilde{R} \hookrightarrow S^1\times \widetilde{R}$ denote the inclusion.  Finally, let $p: S^1\times \widetilde{R}\rightarrow S^1$ be projection to the first factor.

\begin{defn} Given a region $R$, the \dfn{defect} $n(R)$ of the region is 
 \[n(R)=\frac{1}{A_R} \sum_{i=1}^{k_R} \operatorname{deg}(p \circ \iota \circ \gamma_i).\]
\end{defn}

It is immediate from the definition that $n(R)=0$ if and only if the (multi)curve $\widetilde{L}_R$ bounds a section of the $S^1$ bundle.  In fact, this implies that the defect is independent of the chosen trivialization.

% \jlc{To see this, first note that a trivialization determines a first homology class in each boundary component of the (possibly lifted) bundle.  There is representative of this class which intersects an arbitrary fiber in the boundary component once, transversely, and together these determine a basis for the first homology of each torus.  To compute the defect of a region with respect to the chosen trivialization, describe $[\widetilde{L}_R]$ in terms of this basis and sum the coefficients associated to the fiber classes.  A different trivialization will change these coefficients by contributions which sum to zero, showing that the sum is independent of the original choice.}

It follows from this definition that the defect is additive on regions.  When $R$ contains no orbifold points, then the defect $n(R)$ is an integer; in general, the defect contains information about the Euler number of $M_R$.

\begin{lem} \label{lem:orbi-defect} The difference between the defect $n(R)$ and the Euler number $e(M_R)$ is an integer, i.e.\ $n(R) - e(M_R) \in \zz$.
\end{lem}

\begin{proof} Recall that each exceptional fiber $F'$ can be viewed as the core of a solid torus where Dehn surgery was performed on some regular fiber $F$.  Let $K$ be a loop bounding a meridional disc in a regular neighborhood of $F$.  After performing $(\alpha, \beta)$ surgery, $K$ intersects a meridian of the surgered torus $-\beta$ times, so the defect of the region bounded by $K$ is $\frac{-\beta}{\alpha}$.  

Now let $K_1, \ldots, K_l$ be small loops in $R$ around the $l$ exceptional fibers in $M_R$.  The multicurve $\partial R \bigcup (\cup K_i)$ bounds a region with no orbifold points, and hence has an integral defect $d$.  Since the defect is additive, we see that 
$$n(R) = d - \sum_{i=1}^l \frac{\beta_i}{\alpha_i} = d' + e(M_R)$$
for some $d' \in \zz$.
\end{proof}

We say that a diagram $(\Sigma, \Gamma_L)$ is \dfn{labeled} when it is decorated with a defect in each region and with the fiber invariants associated to each orbifold point.  Abusing notation, we will refer to both the projection and the labeled diagram by $\Gamma_L$. Isotopy of the link changes the labeled diagram in one of several ways.  Figure~\ref{fig:moves} shows labeled Reidemeister moves for links in a Seifert fibered space; these correspond to isotopies of $L$ in the complement of the exceptional fibers.  When a strand of $L$ passes through an exceptional fiber of type $(\alpha, \beta)$ the labeled diagram changes by a \dfn{teardrop} move which wraps $\Gamma$ around the orbifold point $\alpha$ times.  See Figure~\ref{fig:tearlabels}.

 \begin{figure}[h!]
\begin{center}
    \relabelbox \tiny{
      \centerline{\scalebox{.5}{\epsfbox{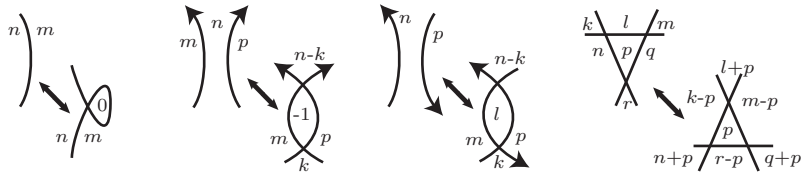}}}}
    \relabel{1}{$n$}
    \relabel{2}{$m$}
    \relabel{3}{$n$}
    \relabel{4}{$m$}
    \relabel{5}{$0$}
    \relabel{6}{$m$}
    \relabel{7}{$n$}
    \relabel{8}{$p$}
    \relabel{9}{$m$}
    \relabel{10}{-$1$}
    \relabel{11}{$p$}
    \relabel{12}{$k$}
    \relabel{13}{$n$-$k$}
    \relabel{14}{$n$}
    \relabel{15}{$p$}
    \relabel{16}{$m$}
    \relabel{17}{$l$}
    \relabel{18}{$p$}
    \relabel{19}{$k$}
    \relabel{20}{$n$-$k$}
    \relabel{21}{$k$}
    \relabel{22}{$l$}
    \relabel{23}{$m$}
    \relabel{24}{$n$}
    \relabel{25}{$p$}
    \relabel{26}{$q$}
    \relabel{27}{$r$}
    \relabel{28}{$l$+$p$}
    \relabel{29}{$k$-$p$}
    \relabel{30}{$p$}
    \relabel{31}{$m$-$p$}
    \relabel{32}{$n$+$p$}
    \relabel{33}{$r$-$p$}
    \relabel{34}{$q$+$p$}
   \endrelabelbox
   \caption{ Labeled Reidemeister moves.}\label{fig:moves}
   \end{center}
\end{figure}

In order to label the new regions created by a teardrop, we assume that the isotopy occurs in an arbitrarily small neighborhood of the exceptional fiber.  The defect is therefore completely determined by the preferred chords at the new crossings.  We may choose a local metric on the solid torus over the neighborhood of an orbifold point so that each regular fiber has length $1$ and the exceptional fiber has length $\frac{1}{\alpha}$.  With such a choice, the chords created by the teardrop have lengths in the set $\{ \frac{1}{\alpha}, \frac{2}{\alpha}, \hdots \frac{\alpha-1}{\alpha}\}$.

\begin{figure}[h!]
\begin{center}
    \relabelbox \footnotesize{
      \centerline{\scalebox{.65}{\epsfbox{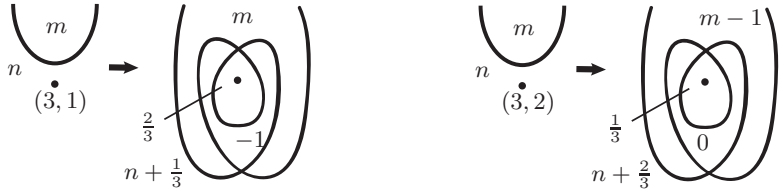}}}}
    \relabel{1}{$n$}
    \relabel{2}{$m$}
    \relabel{3}{$(3,1)$}
    \relabel{4}{$m$}
    \relabel{5}{$-1$}
    \relabel{6}{$n+\frac{1}{3}$}
    \relabel{7}{$\frac{2}{3}$}
    \relabel{8}{$n$}
     \relabel{9}{$m$}
    \relabel{10}{$(3,2)$}
    \relabel{11}{$\frac{1}{3}$}
    \relabel{12}{$n+\frac{2}{3}$}
    \relabel{13}{$0$}
     \relabel{14}{$m-1$}
     \endrelabelbox
  \caption{Labeled teardrop moves for $\alpha=3$. Left: Since the defect of the innermost region is $\frac{k}{3}$ with $k\equiv -1$ modulo $3$, the length of the innermost chord is $\frac{2}{3}$.  The defect of the next-innermost region is an integer, so the length of the other preferred chord satisfies $-2-2+j\equiv 0$ modulo $3$.  Right: The inner chord has length $\frac{1}{3}$ and the outer chord has length $\frac{2}{3}$. }\label{fig:tearlabels}
   \end{center}
\end{figure}

The defect of a region is the signed sum of the lengths of the chords assigned to its corners, where the sign is positive at coherent corners and negative otherwise.  Since the innermost region of the teardrop has a coherent corner, it follows from Lemma~\ref{lem:orbi-defect} that the defect of this region is $\frac{\alpha-\beta}{\alpha}$.  The defects of the other regions are determined by the signs of the corners and the requirement that the defect of any region not containing an orbifold point is integral.

 We say that two labeled diagrams are \dfn{equivalent} if they differ only by sequence of surface isotopies in $\Sigma$, labeled Reidemeister moves, or labeled teardrop moves.  The discussion above, together with the classical Reidemeister theorem, establishes the following lemma:

\begin{lem} \label{prop:diagram-isotopy}
  If two generic links in $M$ are isotopic, then their labeled diagrams are equivalent. 
 \end{lem}

\begin{figure}[h!]
\begin{center}
    \relabelbox \footnotesize{
      \centerline{\scalebox{.6}{\epsfbox{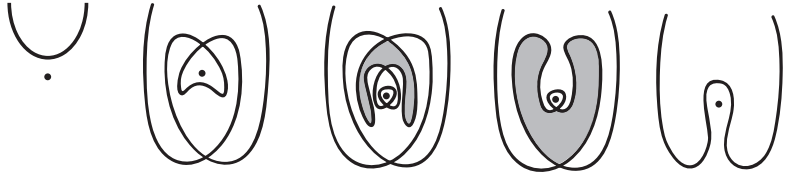}}}}
   % \relabel{1}{n}
    %\relabel{2}{m}
    %\relabel{3}{(3,1)}
    %\relabel{4}{m}
    %\relabel{5}{-1}
       % \relabel{6}{$n+\frac{1}{3}$}
    %\relabel{7}{$\frac{2}{3}$}
    %\relabel{8}{-1}
      %\relabel{9}{m}
    %\relabel{10}{1}
    %\relabel{11}{-1}
    %\relabel{12}{$\frac{2}{3}$}
    %\relabel{13}{$n+\frac{1}{3}$}
       % \relabel{14}{$n+\frac{1}{3}$}
    %\relabel{15}{$\frac{2}{3}$}
    %\relabel{16}{m-1}
    %\relabel{17}{1}
    %\relabel{18}{n}
    %\relabel{19}{m}
     \endrelabelbox
  \caption{Composing two teardrop moves with Reidemeister II moves (through the shaded regions) returns a diagram isotopic to the original.}\label{fig:octo}
   \end{center}
\end{figure}

\begin{rem} Although it is possible to define an inverse for the teardrop move, we present it as unidirectional; passing the innermost  strand of the teardrop back across the fiber introduces a second teardrop, and a sequence of Reidemister  II moves returns a projection isotopic to the original one.  See Figure~\ref{fig:octo} for an example. 
\end{rem}

Next, we define a diagram move that preserves $\Gamma$ but alters the defects of a pair of adjacent regions.  Following Turaev, we say \dfn{fiber fusion} is the operation that replaces an oriented segment of $L$ with a segment that has the same projection but travels once around the fiber. 

We define an action of $H_1(\Sigma)$ on the set of oriented links $\mathcal{L}(M)$ as follows:  let $\gamma$ be a generic simple closed curve on $\Sigma$ that represents a class $[\gamma] \in H_1(\Sigma)$; in particular, we assume that $\gamma$ intersects $\pi(L)$ transversely in finitely many points and misses the double points of $\pi(L)$ and the orbifold points of $\Sigma$.  Construct the link $\gamma \cdot L$ by performing fiber fusion on $L$ in a neighborhood of each point of $\gamma \cap \pi(L)$, where the sign of intersection dictates the sign of the fusion.

\begin{lem} \label{lem:h1-action}
  The isotopy type of the link $\gamma \cdot L$ depends only on the homology class $[\gamma]$. \end{lem}
  
\begin{proof} The proof is the same as that in \cite{turaev:shadow}.
\end{proof}

We note that the labeled diagrams associated to $L$ and to $\gamma\cdot L$ have the same defects; this follows from the fact that for each region $R$, the closed loop $\gamma$ intersects  $\partial R$ zero times algebraically. Consequently,  a labeled diagram of genus greater than zero cannot determine an isotopy class of link.  We show next that each labeled diagram corresponds to  an equivalence class of links related by this $H_1(\Sigma)$ action.  

Let $\bar{\alpha}=(\alpha_1, \alpha_2, \hdots \alpha_k)$ be a list of the orders of the orbifold points on $\Sigma$.  Pick a list $\bar{\beta}=(\beta_1, \beta_2, \hdots \beta_k)$ such that $(\alpha_i, \beta_i)$ are relatively prime and $1\leq \beta_i<\alpha_i$.   Let $\mathcal{D}(\Sigma,q,\bar{\alpha}, \bar{\beta})$ denote the set of labeled diagrams  whose defects sum to $q$ and satisfy Lemma~\ref{lem:orbi-defect} in each region.

\begin{thm} \label{thm:realization} Let $M$ be a Seifert fibered space with exceptional fiber invariants $\{ (\alpha_i, \beta_i)\}$. There is a bijective correspondence between the set $\mathcal{D}(\Sigma, e(M), \mathbf{\alpha}, \mathbf{\beta})$, up to equivalence, and the set $\mathcal{L}(M)$, up to isotopy and the action of $H_1(\Sigma)$.  \end{thm}

In the absence of exceptional fibers, we note that this result follows from a theorem of Turaev which establishes a bijection between his ``shadow links'' and isotopy classes of links in $M$, up to the action of $H_1(\Sigma)$. To see the theorem in this special case, we describe a bijection between labeled diagrams and shadow links.  Let $R$ be a region of $\Sigma \setminus \Gamma$ with $p(R)$ positive corners and $q(R)$ negative corners with respect to the preferred chords.  In the notation of \cite{turaev:shadow},  $\alpha=2n(R)-p(R)+q(R)$ and $\beta=p(R)+q(R)$.  It follows that the ``gleam'' of $R$ is $p(R)-n(R)$.

\begin{proof}[Proof of Theorem~\ref{thm:realization}] As a first step, we show for a given labeled diagram in $\mathcal{D}(\Sigma, e(M), \mathbf{\alpha}, \mathbf{\beta})$, one may always find a link $L$ realizing this diagram.  

Fix a labeled diagram $(\Sigma, \Gamma) \in \mathcal{D}(\Sigma, e(M), \mathbf{\alpha}, \mathbf{\beta})$.  One may easily find a link $L$ in $M$ which projects to $\Gamma$, and by Lemma~\ref{lem:orbi-defect}, the defect of any region will differ from the Euler number of the bundle over that region by an integer.  We induct on the number of crossings to show that $L$ may be modified so that its defect in each region agrees with the given label.  For the base case, consider a diagram consisting of a collection  of disjoint embedded circles.  Selecting an arbitrary component of $\Sigma \setminus \Gamma$ to be ``outermost'' gives a partial order on the components of $\Gamma$.  Perform fiber fusions on the curves of $L$ which project to the boundary of any innermost region in order to adjust its defect to the given label.  Proceed outward, region by region.  Upon reaching the outermost region, there will be no free edges available for fiber fusion, but since each fusion operation preserves the sum of the labels, the defect of the outermost region will automatically agree with the given label. 

Now suppose that for any labeled diagram with fewer than $n$ crossings, we can find a knot $L\subset M$ whose defects agree with the labels.  Let   $(\Sigma, \Gamma)\in \mathcal{D}(\Sigma, e(M), \mathbf{\alpha}, \mathbf{\beta})$ have $n$ crossings.  Resolve one crossing so as to preserve the orientation of $\Gamma$ and apply the inductive hypothesis to construct a link $L'$ whose defects agree with the labels.  Replacing the crossing splits one region into two pieces, and Figure~\ref{fig:adjust} indicates how to perform fiber fusions to construct the desired $L$.

\begin{figure}[h!]
\begin{center}
    \relabelbox \footnotesize{
      \centerline{\scalebox{.55}{\epsfbox{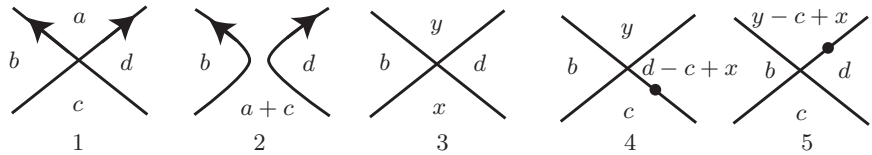}}}}
    \relabel{1}{$1$}
    \relabel{2}{$2$}
    \relabel{3}{$3$}
    \relabel{4}{$4$}
    \relabel{5}{$5$}
        \relabel{a}{$a$}
    \relabel{b}{$b$}
    \relabel{c}{$c$}
      \relabel{d}{$d$}
    \relabel{10}{$b$}
    \relabel{11}{$d$}
    \relabel{a+c}{$a+c$}
    \relabel{13}{$b$}
        \relabel{14}{$d$}
    \relabel{y}{$y$}
    \relabel{x}{$x$}
    \relabel{17}{$y$}
    \relabel{18}{$b$}
    \relabel{19}{$c$}
     \relabel{20}{$d-c+x$}
    \relabel{21}{$b$}
        \relabel{22}{$c$}
    \relabel{23}{$d$}
    \relabel{y-c+x}{$y-c+x$}
     \endrelabelbox
  \caption{ Given a labeled diagram (1), resolve a crossing of $(\Sigma, \Gamma)$ in order to apply the inductive hypothesis (2).  Replace the crossing (3), noting that $x+y=a+c$.  Finally, perform fiber fusions to $L'$ until its defects are as desired (4, 5).}\label{fig:adjust}   \end{center}
\end{figure}

As in \cite{turaev:shadow}, the remainder of the proof of Theorem~\ref{thm:realization} follows from two further steps.  The first step is showing that any two isotopy classes of links which correspond to the same labeled diagram are related by the action of $H_1(\Sigma)$.  The second step establishes that two generic links corresponding to equivalent labeled diagrams are related by a sequence of fiber fusions and isotopies.  Turaev's arguments apply with little modification to both cases; in the second case, we additionally note that any teardrop move on  labeled diagrams can be realized by a local isotopy of the link across an exceptional fiber. 
\end{proof}

% **********
\section{Combinatorics for Rational Seifert Surfaces}
\label{sec:frss}

In this section, we develop a combinatorial description of a rational Seifert surface for a rationally null-homologous knot $K$.  The description has the form of two decorations of the labeled diagram $\Gamma_K$ of $K$:  a ``formal rational Seifert surface'' and a compatible ``fiber distribution''.  The two decorations will be used in the next section to describe a generalization of the Seifert algorithm.

% *****
\subsection{Two Decorations of Labeled Diagrams}

A surface in a Seifert fibered space is said to be horizontal if it is everywhere transverse to the fibers; we relax this condition slightly and consider rational Seifert surfaces which are transverse except near fibers over double points of $\Gamma$.  The idea of the first decoration is that any such surface assigns a multiplicity to each region $R$. Conversely, we may characterize the sets of multiplicities on $\Sigma$ which are induced by such a surface using the following combinatorial object:

\begin{defn}\label{def:rfss} A \dfn{formal rational Seifert surface $\mathbf{m}$ of order $r$} is an assignment of an integral multiplicity $m(R_j)$ to each region  $R_j$ of $\Sigma\setminus \Gamma$ which satisfies the following conditions:
\begin{enumerate}
\item\label{except} The least common multiple $A_{R_j}$ of the orders of the orbifold points in $R_j$ divides $m(R_j)$;
\item\label{order} if $R_k$ and $R_l$ share an edge oriented as $\partial R_k$, then \[m(R_k)-m(R_l)=r;\]
\item \label{total}  summing over all regions, \[ \sum_j m(R_j) n(R_j)=0.\]
\end{enumerate}
\end{defn}

A formal rational Seifert surface may be viewed as a secondary labeling on a knot diagram, and we introduce a tertiary labeling as well. Let $x_j^i$ denote a corner of the region $R_j$ at the $i^{th}$ crossing.  (It is possible for a single region to fill more than one corner at a given crossing, but for notational convenience, we avoid introducing a third index to distinguish them.)

\begin{defn} \label{defn:fiber-dist-g} 
 Given a formal rational Seifert surface $\mathbf{m}$ for a labeled diagram $\Gamma$, a \dfn{fiber distribution} compatible with $\mathbf{m}$ is an assignment $\mathbf{f}$ of integers $f(x_j^i)$ to the corners of regions of $\Sigma \setminus \Gamma$ which satisfies the following properties:
  \begin{enumerate} 
  \item \label{flatdiscg} for each region $R_j$ with corners $x^i_j$ for $i\in \mathcal{C}_R = \{i_1, \ldots, i_{k_R}\}$, \[ m(R_j) n(R_j)+\sum_{i \in \mathcal{C}_R} f(x_j^i)=0;\]
  \item \label{welllabg} for each crossing labeled $i$ with incident regions $R_{j_1}, \ldots, R_{j_4}$, \[ \sum_{k=1}^4 f(x_{j_k}^i) =0.\]
 \end{enumerate} 
\end{defn}

Rational formal Seifert sufaces and their fiber distributions are best understood in terms of a special cell decomposition of $M$, which is constructed in Section~\ref{sec:cell}.  As motivation, however, one may view the rational formal Seifert surface as describing how a surface interacts with the base orbifold $\Sigma$, whereas a fiber distribution captures its interaction with the bundle structure of $M$.

\begin{exam}\label{ex:1}
The figure shows a labeled diagram for a knot in $L(5,2)$, together with a rational formal Seifert surface and fiber distribution.    

\begin{figure}[h!]
\begin{center}
    \relabelbox \footnotesize{
      \centerline{\scalebox{.8}{\epsfbox{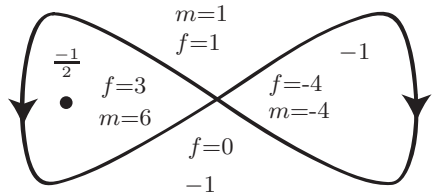}}}}
    \relabel{x}{$-1$}
    \relabel{ba}{$\frac{-1}{2}$}
    \relabel{-1}{$-1$}
      \relabel{m1}{$m$=$1$}
    \relabel{m2}{$m$=$6$}
    \relabel{m4}{$m$=-$4$}
    \relabel{f1}{$f$=$1$}
    \relabel{f2}{$f$=$3$}
    \relabel{f3}{$f$=$0$}
        \relabel{f4}{$f$=-$4$}
     \endrelabelbox
  \caption{ A labeled diagram for $K\subset L(5,2)$, together with a formal rational Seifert surface and a compatible fiber distribution.}\label{fig:lensex} \end{center}
\end{figure}

We will use this example to illustrate the generalized Seifert algorithm in Section~\ref{sec:seifert-algorithm}.

\end{exam}

\subsection{A cell decomposition for $M$}\label{sec:cell}

In this section, we construct a cell decomposition of $M$ using data from the knot $K$. We begin by enlarging the graph $\Gamma_K$ so that each complementary region is homeomorphic to a disc and contains at most one orbifold point.  If a region has nontrivial topology or contains more than one orbifold point, subdivide it using a collection of arcs $\Gamma_0 \subset \Sigma$ whose endpoints lie on $\Gamma_K$; let $\bar{\Gamma}$ denote the graph $\Gamma_K \cup \Gamma_0$.   Lift the arcs of $\Gamma_0$ to curves $K_0$ in $M$ whose endpoints lie on $K$. The knot $K$, the arcs $K_0$, and the fibers over each vertex of $\bar{\Gamma}$ form a $1$-complex in $M$.

The $2$-skeleton of $M$ consists of two types of cells.  First, for each edge $e$ of $\bar{\Gamma}$, let $D_e$ be the preimage of $e$ in $M$,  thought of as a disc whose boundary consists of the fibers over the ends of the edge, together with two oppositely-oriented copies of the corresponding segment of $K \cup K_0$. Refer to this type of cell as \dfn{vertical}.  Second, for each region $R$ of $\Sigma \setminus \bar{\Gamma} $, we construct the \dfn{regional} cell $D_R$ as follows.  Denote the fibers over double points in $\partial R$ by $\{F_i\}$.   The lifted curve $K_{R}$ satisfies $[A_R K_{R}-\sum b_i F_i]=0\in H_1(M)$ for any  $b_i$ such that $\sum b_i=A_R n(R)$. The $1$-chain $A_R K_{R}-\sum b_i F_i$ bounds a disc in $M_R$, and we include this as the $2$-cell $D_R$.

The remainder of $M$ consists of $3$-balls that come from removing a meridian disc from the solid tori over each region of $\bar{\Gamma}$; these balls make up the $3$-skeleton.

\subsection{Proof of Theorem~\ref{thm:label}}
Recall the statement of Theorem~\ref{thm:label} from the introduction:

\begin{thm1} If $K$ is rationally null-homologous in a Seifert fibered space, then any labeled diagram for $K$ admits a formal rational Seifert surface with a compatible fiber distribution.
\end{thm1}

\begin{proof}
Suppose that $K$ is rationally null-homologous with order $r$.  The knot $K$ has an obvious representative (which we shall also call $K$) as a $1$-chain in the cell decomposition described above.  Hence, there exists a $2$-chain $S$ such that $\partial S = rK$. For each region $R_j\in \Sigma\setminus (\bar{\Gamma} )$,   let $c_j$ denote the coefficient of $D_j$ in $S$.  Assign the multiplicity  $m(R_j)$ to be $c_j\alpha_{R_j}$. 

We begin by verifying Condition~\ref{except} of Definition~\ref{def:rfss}.  It is clearly satisfied on disc components of $\Sigma \setminus \bar{\Gamma}$.  Now suppose that $R_1$ and $R_2$ in $\Sigma\setminus \bar{\Gamma})$ are separated by the edge $e_0\in \Gamma_0$.  The assumption that $\partial S=rK$ implies that this edge has multiplicity $0$ in $\partial S$, so $m(R_1)=m(R_2)$.  This shows that the multiplicities are well-defined on components of $\Sigma \setminus \Gamma_K$. Since $\alpha_j \ | \ m(R_j)$ for $j=1,2$ and $m(R_1)=m(R_2)$, Condition~\ref{except} is satisfied on $R_1\cup R_2$, and an inductive argument shows that it holds for all components of $\Sigma \setminus \Gamma_K$.

Adding a vertical $2$-cell to a chain does not change the coefficient of any edge of $K$ in the boundary $1$-chain.  Each edge of $K$ appears $r$ times in $\partial S$, so the difference in multiplicities between the two adjoining regional cells is $r$, establishing Condition~\ref{order}.

Finally, we show that Condition~\ref{total} of Definition~\ref{def:rfss} holds.  By construction, the boundary of each regional cell consists of $A_{R_j}K_{R_j}$ and $-A_{R_j} n(R_j)$ copies of the fiber.  Thus the total number of copies of the fiber coming from regional $2$-cells is $\sum_j -m(R_j)n(R_j)$.  The addition of any vertical $2$-cell preserves this sum, and the assumption that $\partial S=rK$ implies that the copies of the fiber must cancel algebraically: $\sum_j m(R_j)n(R_j)=0$.

To construct a compatible fiber distribution $\mathbf{f}$, consider a quadrant $x^i_j$ of a crossing $i$ lying in the region $R_j$.  Suppose that this quadrant lies to the right of the oriented edges $\mathcal{E}(x^i_j)$ of $\bar{\Gamma}$; note that this set may be empty and has at most two elements.  We then define $f(x^i_j)$ to be
\[f(x^i_j) = -c_j b_i + \sum_{e \in \mathcal{E}(x^i_j)} \epsilon_e f_e,\]
where the integer $b_i$ comes from the construction of the regional cell $D_{R_j}$, $f_e$ is the coefficient of the vertical cell $D_e$ in $S$, and $\epsilon_e$ is positive if and only if the head of $e$ is incident to the double point $i$.

Condition \ref{flatdiscg} of Definition~\ref{defn:fiber-dist-g} now follows from two facts.  First, observe that $c_j \sum b_i = c_j A_{R_j} n(R_j) = m(R_j) n(R_j)$.  Second, note that each edge with $R_j$ on its right contributes $f_e$ to the sum associated to the quadrant at its head and and $-f_e$ to the sum associated to the quadrant at its foot; thus, the contributions coming from the vertical $2$-cells cancel around any given region.  Condition~\ref{welllabg} holds because $\sum_{k=1}^4 f(x_{j_k}^i)$ is the coefficient of the fiber over the double point $i$ in $\partial S$, but we know that $\partial S = rK$, and hence this coefficient must vanish.
\end{proof}

\begin{rem}
One may show that every formal rational Seifert surface admits a compatible fiber distribution, a fact which permits a stronger formulation of Theorem~\ref{thm:bounds}.  The proof is by induction on the number of double points of $\Gamma$, and we leave the details to the reader.   
\end{rem}

% **********
\section{Seifert algorithm for knots in $S^1$ orbifold bundles}
\label{sec:seifert-algorithm}

Given a formal rational Seifert surface $\mathbf{m}$ and a fiber distribution $\mathbf{f}$ for a rationally null-homologous knot of order $r$, we construct a rational Seifert surface of the same order.  The classical Seifert algorithm for knots in $\rr^3$ proceeds in three steps: first, one resolves the crossings in a projection of the knot to obtain a collection of Seifert circles in the plane.  Second, one views the Seifert circles as bounding disjoint embedded disks.  Finally, the Seifert disks are connected by twisted bands at the crossings.  The generalized algorithm for a knot in a Seifert fibered space parallels the classical algorithm. As a first step, we let $D_i$ denote a neighborhood of the $i^{th}$ double point of $\Gamma$ and let $U_i = \pi^{-1}(D_i)$.  We use \textbf{m} and \textbf{f} to resolve the knot into circles in $M \setminus \bigcup U_i$ (Section \ref{ssec:seifert-circle}).  Next, we view these resolved circles as bounding embedded surfaces in $M \setminus U_i$ (Section \ref{ssec:seifert-surfaces}). Finally, we extend these surfaces across the solid tori $U_i$ (Section \ref{ssec:crossing-construction}).  We begin by establishing notation which will be useful throughout the algorithm.

For each double point of $\Gamma$, parameterize the neighborhood $D_i$ as a unit disc  and let $C^i_t$ denote the $S^1$ bundle over the circle of radius $t$.  Dropping the superscript when the crossing is obvious, we split the torus $C_1$ into annuli denoted $A_{I}$, $A_{II}$, $A_{III}$, and $A_{IV}$ according to the corresponding quadrants of $\Sigma$; see Figure~\ref{fig:notation}.

 \begin{figure}[h!]
\begin{center}
    \relabelbox \footnotesize{
      \centerline{\scalebox{.7}{\epsfbox{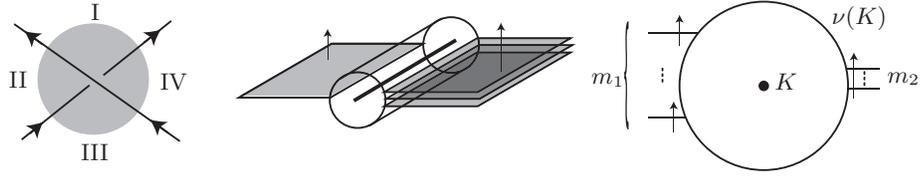}}}}
    \relabel{0}{$\nu(K)$}
    \relabel{1}{$m_1$}
    \relabel{2}{$K$}
    \relabel{3}{$m_2$}
    \relabel{I}{I}
        \relabel{II}{II}
    \relabel{III}{III}
    \relabel{IV}{IV}
   \endrelabelbox
   \caption{Left: The disc $D_i$ near a double point of $\Gamma$.  Center: $S$ in a neighborhood of $K$ away from a double point.  Right: A cross-section of $N$ between regions with multiplicities $m_1$ and $m_2$, where $m_1-m_2=r$.}\label{fig:notation}
\end{center}
\end{figure}

Let $K_0$ be the curves constructed in Section~\ref{sec:cell}.  Near $K \cup K_0$ but away from the double points of $\bar\Gamma$, the local behavior of any rational Seifert surface is dictated by the multiplicities of the adjacent regions; note that the multiplicities on regions of $\Gamma$ induce multiplicities on the regions of $\bar\Gamma$.  Let $N$ be a regular neighborhood of $K \cup K_0$, and suppose that the projection of a segment of $K \cup K_0$ separates regions with multiplicities $m_1$ and $m_2$.  In this case, the rational Seifert surface $S$ intersects $\partial N$ $m_i$ times on each side.  Correspondingly, to each side of a cross section of $N$ we draw $m_i$ parallel, transversely-oriented lines.  The endpoints of these lines trace out $m_i$ parallel curve segments on $\partial N$  as the cross-section varies; see Figure~\ref{fig:notation}.

\subsection{Resolution into Seifert Circles}
\label{ssec:seifert-circle}

The first step of the construction replaces $K$ with a collection of circles.
Remove the interior of $N$ and the fibered solid tori $U_i$ from $M$. As described above, the portions of $\partial N$ away from the $U_i$ and neighborhoods of the intersection points $K \cap K_0$ are decorated with collections of parallel curves.  Near the intersection points $K \cap K_0$, we simply join the endpoints of corresponding parallel curves.  Near the solid tori $U_i$, we will use \textbf{m} and \textbf{f} to construct a pattern of curves on $C^i_1$ which connect the endpoints of the parallel curves.

 \begin{figure}[h!]
\begin{center}
    \relabelbox \footnotesize{
      \centerline{\scalebox{.75}{\epsfbox{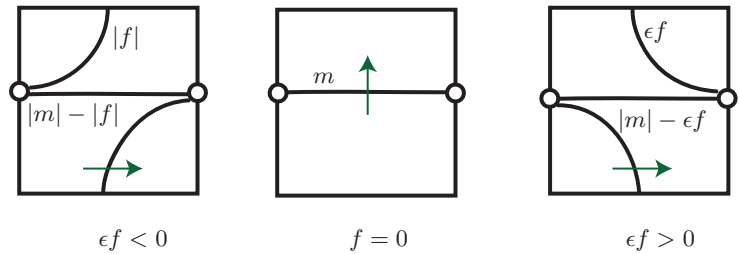}}}}
    \relabel{1}{$|f|$}
    \relabel{2}{$| m|-|f|$}
        \relabel{3}{$m$}
        \relabel{4}{$\epsilon f$}
        \relabel{5}{$|m|-\epsilon f$}
        \relabel{6}{$\epsilon f<0$}
        \relabel{7}{$f=0$}
        \relabel{8}{$\epsilon f>0$}
        \endrelabelbox
   \caption{Left: The figure above shows the local models for $f>0$.  The circles on the side edges are the intersections between $C_1$ and $N$. Changing the sign of $f$ reverses the arrows.}\label{fig:basicpattern}
\end{center}
\end{figure}

Fix a crossing, and for convenience, cut the corresponding solid torus along a meridional disc so that $C_1$ becomes a cylinder composed of four rectangles still labeled by I, II, III, and IV.  Orienting each rectangle as if viewed from $t>1$, decorate it with a pattern of multicurves as shown in Figure~\ref{fig:basicpattern}.  Each curve is decorated with an arrow indicating its transverse orientation and by an integer weight indicating its multiplicity.  Reversing the arrrow changes the sign of this weight. By construction, the endpoints of these curves can be glued to the endpoints of the curves on $\partial N$.  

The resulting pattern of curves on $\partial \big(N \cup (\cup_iU_i)\big)$  will serve as our Seifert circles. Before proceeding, we note the following:

\begin{lem}\label{lem:cancel} The sum of the algebraic intersection numbers of the pattern curves with the meridian of $C_1$ is zero around each double point.
\end{lem}

\begin{proof} This follows from Condition~\ref{welllabg} of Definition~\ref{defn:fiber-dist-g}. 
\end{proof}

\begin{figure}[h]
\begin{center}
    \relabelbox \footnotesize{
      \centerline{\scalebox{.55}{\epsfbox{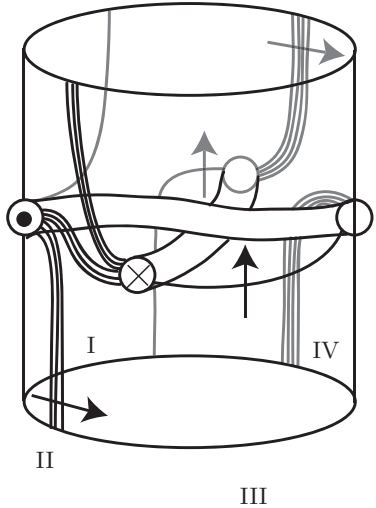}}}}
      \relabel{a}{I}
    \relabel{b}{II}
    \relabel{c}{III}
    \relabel{d}{IV}
     \endrelabelbox
   \caption{The intersection pattern on $C_1$ for the unique crossing in the knot from Example~\ref{ex:1}.}\label{fig:expat}
\end{center}
\end{figure}

\subsection{Surfaces Bounded by Seifert Circles}
\label{ssec:seifert-surfaces}

We begin the second step by constructing surfaces in $M\setminus \big( N \cup (\cup_iU_i)\big)$ bounded by the Seifert circles.

Condition~\ref{flatdiscg} of Definition~\ref{defn:fiber-dist-g} implies that all $\frac{m(R)}{A_R}$ Seifert circles over the boundary of a given region $R$ are null-homologous and hence bound horizontal embedded discs in $M_R$.  By construction, the signed intersection number of each curve pattern with $\partial N \cap C_1$ is $r$; see Figure~\ref{fig:notation}.

% \begin{figure}[h]
%\begin{center}
    %\relabelbox \footnotesize{
      %\centerline{\scalebox{.47}{\epsfbox{localmodelcirc.eps}}}}
    %\relabel{0}{f=1}
    %\relabel{1}{m=1}
    %\relabel{2}{f=1}
    %\relabel{3}{m=-1}
    %\relabel{4}{f=0}
    %\relabel{5}{f=0}
      %\relabel{6}{m=0}
        % \relabel{7}{-1}
         %\relabel{8}{-1}
         %\relabel{9}{-1}
 %\endrelabelbox
  % \caption{Constructing a boundary pattern near a crossing for a knot in $L(3,1)$.}\label{fig:scheme}\end{center}
%\end{figure}

% The proof of Theorem~\ref{thm:label} establishes that the multiplicities of regions adjacent to a curve in $\pi(K_0)$ are equal.  Thus, we may simply extend the horizontal discs constructed above straight across the part of $N$ that surrounds $K_0$.

To complete this step, we extend this surface over the cylinders $N \cap \big(M \setminus \cup  U_i\big)$.  There are two cases to consider for the extension over such a cylinder.  If the multiplicities of the adjoining regions have the same sign, then we extend the embeddings of the surfaces as in Figure~\ref{fig:nbhd-extend}(a) for an appropriate choice of $k,l \geq 0$. In particular, if the regions in question are separated by an edge of $\pi(K_0)$, then the multiplicities of the adjacent regions are the same and we use $k+l = m_-$ in the figure. If, on the other hand, the multiplicities of the adjoining regions have opposite signs, then the extension is as in Figure~\ref{fig:nbhd-extend}(b); in this case, there is no choice to make.

 \begin{figure}[h]
\begin{center}
    \relabelbox \footnotesize{
      \centerline{\epsfbox{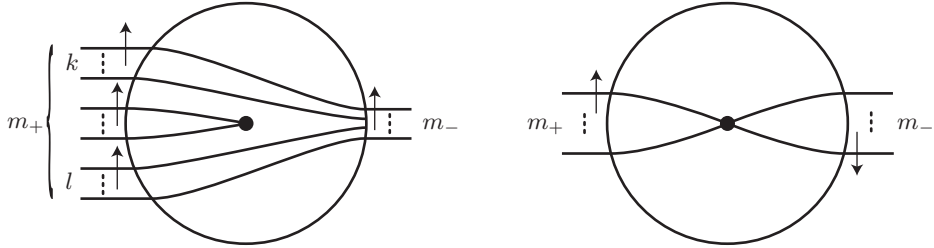}}}
    \relabel{0}{$k$}
    \relabel{1}{$m_+$}
    \relabel{2}{$l$}
    \relabel{3}{$m_-$}
    \relabel{4}{$m_+$}
    \relabel{5}{$m_-$}
   \endrelabelbox
   \caption{The extension of the rational Seifert surface across $N$ if (a) the multiplicities of the adjoining regions have the same sign and (b) if they have opposite signs. }\label{fig:nbhd-extend}
\end{center}
\end{figure}

\subsection{Extending across solid tori over crossings}
\label{ssec:crossing-construction}

We have now constructed a surface in the complement of the crossing tori $U_i$.  In this section, we extend the surface across each $U_i$ by describing how it intersects a collection of concentric cylinders $C_t$ of decreasing radius.  Modifications to the intersection pattern describe changes in the surface.
%Modifications of the intersection pattern that preserve the intersection with $\nu(K)$ correspond to changes in the surface as it extends across the solid torus.\footnote{\jsc{Not sure I understand this sentence.  Preserve the intersection pattern?  Or the signed intersection number? Or\dots?}\jlc{I want to say that we can describe a sequence of intersection patterns on $C_t$ for decreasing values of $t$, and that as long as each one has allowable intersections with the boundary of the neighborhood of $K$, then we're going to get a rational Seifert surface.}}  
In addition to surface isotopy of the curves, we allow the following three primitive moves:  

%Before beginning the extension proper, we lay out the possible modifications we may make to the transversally oriented curves $S \cap C_t$ as $t$ tends to $0$. In addition, we allow the following three primitive moves:
\begin{description}
\item[Finger Moves] We may replace a curve segment adjacent to $C_t\cap N$ with a pair of arcs ending on  $C_t\cap\partial N$; these intersections will have opposite signs.  This move preserves the topology of the surface, but pushes it locally into the neighborhood of $K$. See Figure~\ref{fig:extension-choice}.
\item[Capping a Circle] Any embedded circle may be removed from the intersection pattern.  This corresponds to capping off the corresponding component of $S\cap C_{t_0}$ with a disc embedded in the solid torus defined by $t<t_0$.
\item[Saddle Moves] We may perform a saddle resolution between two curves with opposite transverse oreintations. This corresponds to reducing the Euler characteristic of the surface by $1$.  

\begin{figure}[h]
\begin{center}
\scalebox{.8}{\includegraphics{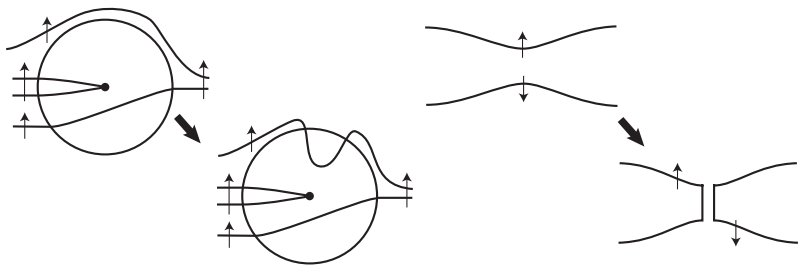}}
\caption{Left: A finger move creates a new pair of intersections between $S$ and $C_t \cap \partial N$.  Right: An oriented saddle resolution.}\label{fig:extension-choice}
\end{center}
\end{figure}

%\item[Finger Moves] We may push the surface $S$ across the top or bottom of\footnote{\jsc{Perhaps too specific\dots}\jlc{I agree. }} $N$, creating two opppositely-signed intersections between $S$ and $\partial N$ in each cross-sectional slice; see Figure~\ref{fig:extension-choice}.
%\item[Saddle Moves] We may perform saddle resolutions between two curves with oppositve transverse oreintations.
\end{description}

We will also make use of two consequences of these three moves.

\begin{description}
\item[Cancellation of Parallel Strands]  Two oppositely-oriented adjacent parallel strands between components of $N \cap C_t$ may be removed. See Figure~\ref{fig:derived}.
\item[Reconfiguration in $N$]  Any two configurations that appear in Figure~\ref{fig:nbhd-extend} are related by a sequence of saddle moves. See Figure~\ref{fig:derived}.
\end{description}

\begin{figure}[h]
\begin{center}
\scalebox{.8}{\includegraphics{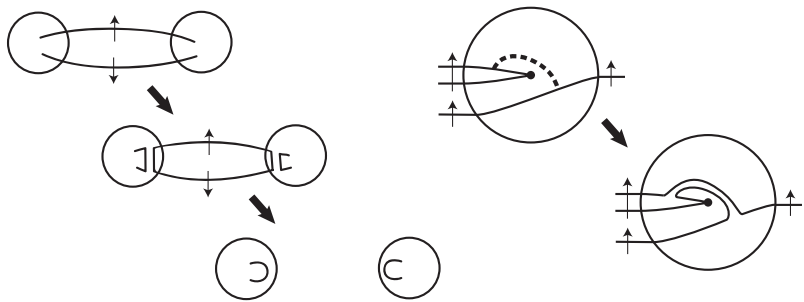}}
\caption{Left: Parallel strands with opposite orientations cancel.  Right:  An oriented saddle move modifies the configuration inside $N$.}\label{fig:derived}
\end{center}
\end{figure}

We now begin to extend the surface $S$ across the solid torus $U_i$.  Isotope all the intersections of the Seifert circles to the annulus $A_{II}$.  Fixing these intersections, standardize the pattern of curves on $C_t$ via isotopy, finger moves, and cancellations of oppositely-oriented parallel strands. Note that after cancellation, the configurations inside $N \cap C_t$ are again of the form in Figure~\ref{fig:nbhd-extend}.  Lemma~\ref{lem:cancel} states that the algebraic intersection number of these curves with the meridian is zero, and saddle resolutions between oppositely-oriented curves reduce the geometric intersection number to zero as well.

The resulting pattern may contain curves with both endpoints  on the same component of $N \cap C_t$; these may be again be removed using sequences of the moves above, especially capping circles.

As $t\rightarrow 0$, the strands of $K$ cross; this rotates a region containing two components of $N \cap C_t$ by $\pi$.   Further finger moves, cancellations, and isotopies yield a standard pattern consisting solely of horizontal curves.  It is clear that these bound a collection of discs, completing $S$.  Note that reconfigurations inside $N$ allow us to match those configurations coming from opposite sides of the intersection of one component of $N \cap U_i$.  See Figure~\ref{fig:torusexample2} for an example.  

\begin{figure}[h!]
\begin{center}
    \relabelbox \footnotesize{
      \centerline{\scalebox{.7}{\epsfbox{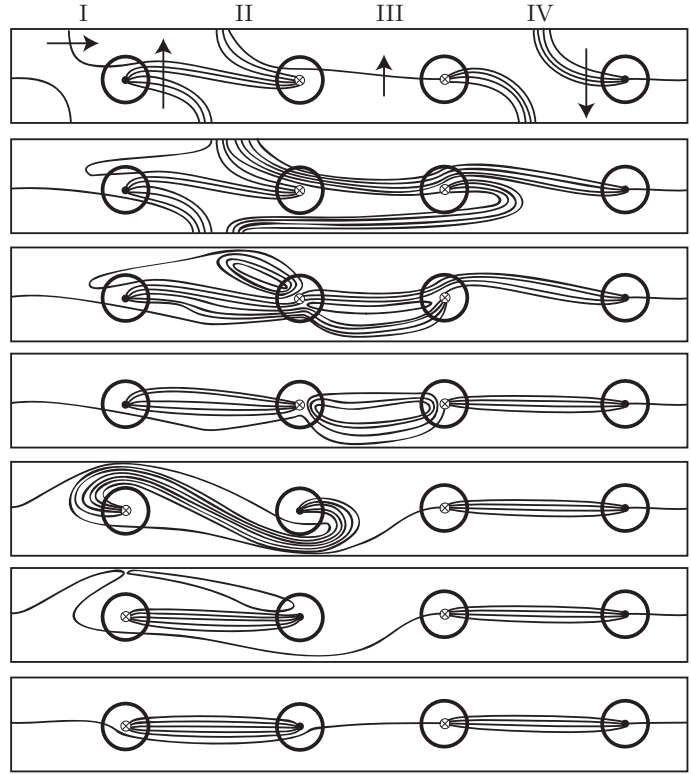}}}}
    \relabel{I}{I}
    \relabel{II}{II}
    \relabel{III}{III}
    \relabel{IV}{IV}
        \endrelabelbox
   \caption{Continuing Example~\ref{ex:1}, this shows a sequence of intersection patterns for decreasing values of $t$.}\label{fig:torusexample2}
\end{center}
\end{figure}

\section{Legendrian invariants}
\label{sec:legendrian}

In this section we use the use the generalized Seifert algorithm to compute the rational classical invariants for a Legendrian knot from a formal rational Seifert surface and fiber distribution.

\subsection{Contact Seifert fibered spaces}\label{sec:legendrian:intro}

We will use the phrase \dfn{contact Seifert fibered space} to denote an orientable Seifert fibered space over an orientable base, equipped with a contact structure $\xi$ transverse to the Seifert fibers.  Such a contact structure exists whenever the rational Euler number of a Seifert fibered space is negative \cite{kt:cr-seifert, lisca-matic:transverse}.  If we further specify a contact form $\alpha$ for $\xi$ with the property that its Reeb field points along the fibers (see \cite{joan-josh:lch4sfs}), then the defect defined in Section~\ref{sec:labeled-diagrams} can be interpreted as an integral of the curvature form associated to $\alpha$ on the Reeb orbit space.  We note that the Legendrian condition precludes the Reidemeister I move of Section~\ref{sec:labeled-diagrams}.

A formal rational Seifert surface $\mathbf{m}$ and a compatible fiber distribution $\mathbf{f}$ may be used to compute may be used to compute the rational classical invariants of a Legendrian knot in a contact Seifert fibered space.  We prove this using  the rational Seifert surfaces constructed in Section~\ref{sec:seifert-algorithm}.

For each region $R_j\in \Sigma\setminus \Gamma$, let $\chi_{orb}(R_j)$ denote the orbifold Euler characteristic of $R_j$ as a sub-orbifold of $\Sigma$; recall that this quantity is defined to be:
\begin{equation} \label{eq:orb-euler-char}
  \chi_{orb}(R) = \chi(R) + \sum_{j=1}^r \left( \frac{1}{\alpha_j} - 1 \right).
\end{equation}
Let $k_j$ and $l_j$ denote the number of double points of $\Gamma$ where $R_j$ fills one or three quadrants, respectively.  We restate Proposition~\ref{prop:initial-leg} as follows:

\begin{prop}\label{prop:legendrian} 
  The rational classical invariants of a null-homologous Legendrian knot $K$ maybe be computed from a formal rational Seifert surface $\mathbf{m}$ and a compatible fiber distribution $\mathbf{f}$ using the following formulae:
\begin{align} \label{eq:r} \text{rot}_{\mathbb{Q}}(K)&=\frac{1}{r}\sum_{\text{regions }R_j} m(R_j) \big[\chi_{orb}(R_j)+\frac{1}{4}(l_j-k_j)\big], \\
\label{eq:tb}
\text{tb}_{\mathbb{Q}}(K)&=\frac{1}{r}\sum_{\text{dble pts }i} (-r-f^i_{II}+f^i_{IV}).
\end{align}
\end{prop}

The subsequent sections discuss these invariants and develop proofs of these propositions.

\begin{exam} The knot in Example~\ref{ex:1} can be realized as a Legendrian knot whose Lagrangian projection is shown in Figure~\ref{fig:lensex}. To see this, begin with the unknot with maximal Thurston-Bennequin number in the standard contact $S^3$.  Performing $1$ and $\frac{1}{2}$ surgery on a pair of regular fibers yields the labeled diagram of Figure~\ref{fig:lensex}, and the contact form may be extended across the surgery tori so that the induced Reeb orbits are the Seifert fibers.

The results above show that this knot has rational rotation number
\[ rot_\qq(K)= \frac{1}{5}\left[6(\frac{1}{4})+1(\frac{1}{2})-4(\frac{3}{4})\right] =-\frac{1}{5}\]
and rational Thurston-Bennequin number 
\[ tb_\qq (K)=\frac{1}{5}(-5-3-4)=\frac{-12}{5}.\]
\end{exam}

\subsection{The rational rotation number}

In \cite{be:rational-tb}, Baker and Etnyre define the rational rotation number of a rationally null-homologous knot by analogy with the classical rotation number for a null-homologous knot.  Let $j:S\hookrightarrow M$ be a rational Seifert surface for $K$.  Trivialize the pulled back contact bundle $j^*\xi$ over $S$ using a nonvanishing vector field $v$; since $K$ is Legendrian, $TK$ lies in the restriction of $\xi$ to $\partial S$.  One may therefore define the winding number of $j^*TK$:
\[ \text{rot}_{\mathbb{Q}}(K)=\frac{1}{r} \text{wind}_V(j^*TK).\]

To better understand a trivialization of $j^*\xi$, we will cut $S$ along its intersection with the vertical tori $\partial U_i$.  This creates a collection of disjoint surfaces with boundary, denoted  collectively by $\hat{S}$; we compute the rotation of each component individually and sum them to compute the rational rotation number of $K$.  Note that cutting introduces new segments to the boundary curves; although these could be isotoped to be Legendrian, their contributions to the rotation will cancel under gluing.  We may therefore ignore these segments and compute only the contributions to the rotation number of $T(\partial \hat{S})$ by $TK$.

We begin by showing that the contribution of a component $X$ of $\hat{S}$ lying in the solid torus $U_i$ to $\text{rot}_{\mathbb{Q}}(K)$ is zero. We may assume that the complex structure on $\Sigma$ is chosen so that the arcs of $\Gamma$ intersect the boundary of the neighborhood of the double point orthogonally.  Choosing the neighborhood of a fixed double point small enough, we may trivialize $T\Sigma$ over the disc $D_i$ with vector fields $\{v, iv\}$ so that $TK$ never coincides with the lines spanned by $v$ and $iv$. Pull back this trivialization to $\xi|_{U_i}$, and then again to $j^*\xi|_{X}$.  With respect to this trivialization, it is obvious that $K \cap U_i$ contributes zero to the rotation number.

We now turn to the portions of $S$ constructed from Seifert circles in Section~\ref{ssec:seifert-surfaces}, i.e., the components of $j(S) \cap (M_R \setminus \bigcup U_i)$.  Recall that these components of $\hat{S}$ are horizontal, and hence that we may identify $TS$ and $j^*\xi$ on these portions. The next lemma extends the existing trivialization of $j^*\xi$ from $j(S) \cap \partial U_i$ and describes the contribution to $\text{rot}_{\mathbb{Q}}(K)$ coming from a single region $R$.

\begin{lem} Suppose that the region $R$ has multiplicity $m(R)$ in a formal rational Seifert surface for $K$, and that $k$ and $l$ donote the number of double points in $\partial R$ where $R$ fills one and three quadrants, respectively.  The contribution of $j(S) \cap (M_R \setminus \bigcup U_i)$ to $\text{rot}_{\mathbb{Q}}(K)$ is  $\frac{1}{r}m(R)\big[\chi_{orb}(R)+\frac{1}{4}(l-k)\big]$.  
\end{lem}

Note that, together with the discussion above, this lemma finishes the proof of the first part of Proposition~\ref{prop:legendrian}.

\begin{proof}As a consequence of trivializing $\xi$ over the solid tori $U_i$, each truncated region may be replaced by the original region without affecting its contribution to the rotation. 

% See Figure~\ref{fig:truncate}.\footnote{\jsc{Are we sufficiently convinced by this to shorten the discussion and/or remove the figure?}}

%  \begin{figure}[ht]
%  \begin{center}
%  \scalebox{.5}{\includegraphics{triv2.eps}}
%  \caption{ Truncating a region near a double point has no effect on its contribution to the rational rotation number of $K$. }\label{fig:truncate}
%  \end{center}
%  \end{figure}
 
Let $S_R$ be a component of $j(S) \cap M_R$.  Note that $S_R$ an $A_R$-fold branched cover of $R$, branched over the orbifold points of $R$. We represent a trivialization of $\xi|_{S_R}$ by a non-vanishing vector field in $T S_R$, and we use the Poincar\'e-Hopf Theorem to compute the winding number of $T \partial S_R$ with respect to this framing on the boundary. 
Embed $S_R$ as a subsurface of a closed surface $\bar{S}_R$ satisfying $\chi(\bar{S}_R)=\chi(S_R)+1$.  Choose a vector field $v$ on $\bar{S}_R$ that extends the trivialization of $\xi$ in the tori $U_i$ and which has the property that its unique critical point $c$ lies in $\bar{S}_R \setminus S_R$.  Because $S_R$ is a branched cover of $R$, we may use the Riemann-Hurwitz Theorem to compute the Euler characteristic of $S_R$:
 \[ \chi(S_R)= A_R\big[\chi(R)+\sum_{i=1}^r(\frac{1}{\alpha_i}-1)\big].\] 
The Poincar\'e-Hopf Theorem implies that the index of $v$ at the unique critical point $c$  is
 \begin{equation}\label{eq:gr1}
 \text{ind}_c v= 1+ A_R\big[ \chi(R)+\sum_{i=1}^r(\frac{1}{\alpha_i}-1)\big].\end{equation}
 
%We need a trivialization of $\xi$ over all of $S_R$.  We have a trivialization which is fixed near the solid tori, but this is inconsequential, since it's local.  So, let's just find some trivialization over the entire surface.  Where we can identify $\xi$ with $T\Sigma$, we can use the page framing on $\Sigma$, but this fails at orbifold points.  Instead, we'd like to use the fact that $S_R$ is horizontal to identify $\xi$ with $TS_R$.  Then, it suffices to find any non-vanishing vector field in $TS_R$.  So, previously, we were focusing on the orbifold points as the points where the identification between $TS$ and $T\Sigma$ failed. 

We now compute the winding number of $\partial S_R$ as an embedded curve with corners which encircles the singular point of the vector field.  For simplicity, consider the curve $-\partial S_R$ (which bounds  a neighborhood of the critical point positively). Identifying this neighborhood with a neighborhood of the origin in $\mathbb{C}$, and compute the winding number of the tangent to $-\partial S_R$ with respect to the translation-invariant page framing:
 
  \begin{equation}\label{eq:gr2} \text{wind}_{page}(-\partial S_R)-A_R(\frac{k}{4}) +A_R(\frac{l}{4})=1.\end{equation}

\begin{figure}[h!]
\begin{center}
    \relabelbox \footnotesize{
      \centerline{\scalebox{.7}{\epsfbox{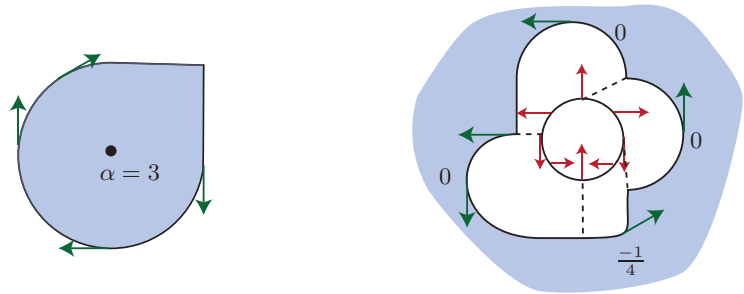}}}}
    \relabel{1}{$\alpha=3$}
    \relabel{2}{$0$}
    \relabel{3}{$0$}
    \relabel{4}{$0$}
        \relabel{5}{$\frac{-1}{4}$}
        \endrelabelbox
   \caption{Left: The boundary of the shaded region is oriented as $-\partial R$, with $k=1$ and $\alpha=3$.  Right: The boundary of a $3$-fold branched cover of $R$ embedded on a sphere.  In a neighborhood of the index two critical point,  $\text{wind}_{page}(-\partial S_R)=\frac{7}{4}$ and $\text{wind}_v(-\partial S_R)=\frac{-1}{4}$.}\label{fig:rotation}
\end{center}
\end{figure}

To convert the winding number with respect to the page framing to the winding number with respect to $v$, subtract the index of $c$:
 \[ \text{wind}_v(-\partial S_R)=\text{wind}_{page}(-\partial S_R)-\text{ind}_c v.\]
 
 The Seifert surface is constructed locally using $\frac{m(R)}{A_R}$ copies of $S_R$, so  the result follows from Equations~(\ref{eq:orb-euler-char}), (\ref{eq:gr1}), and (\ref{eq:gr2}).
 \end{proof}

% ***** 
\subsection{The rational Thurston-Bennequin number}
\label{sect:tb}

In this final section, we use a rational formal Seifert surface and a fiber distribution to compute the rational Thurston-Bennequin invariant.  Recall from \cite{be:rational-tb} that 
the rational Thurston-Bennequin number of a Legendrian knot $K$ is defined to be the rational linking number of $K$ with a transverse push-off $K'$ with respect to some rational Seifert surface for $K$.

Since the fibers are transverse to the contact planes, we may take $K'$  to be the Legendrian push-off  along the Reeb direction; we may think of $K'$ as lying at the bottom of $\partial N$.  Away from the double points of $\Gamma$, the conventions for how a rational Seifert surface $S$ interacts with $N$ in Figure~\ref{fig:notation} imply that there will be no intersection points.  Thus, computing $\text{tb}_{\mathbb{Q}}(K)$ reduces to counting intersections between $S$ and $K'$ in the solid tori over the double points of $\Gamma$.

% The boundary of the rational Seifert surface $S$ lies on $K$, and the intersection of $S$ with a tubular neighborhood $N$  is an embedded curve on $\partial N$. The Legendrian push-off $K'$ also lies on $\partial \nu(S)$, and passing $\partial S$ across the point representing $K'$ in a cylindrical Morse slice corresponds to an intersection between $K'$ and $S$.  The sign of this intersection depends on the transverse orientation of $S$ and the direction of motion, as indicated in Figure~\ref{fig:intersect}. 

\begin{figure}[ht]
 \begin{center}
   \relabelbox \footnotesize{
     \centerline{\scalebox{.7}{\epsfbox{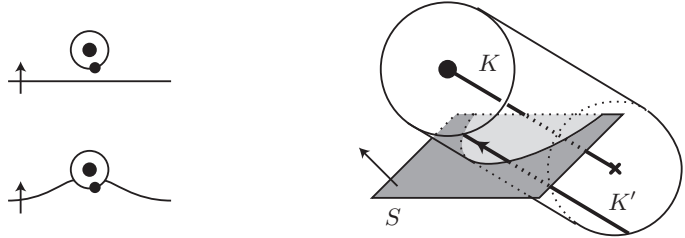}}}}
   \relabel{0}{$S$}
   \relabel{1}{$K'$}
   \relabel{2}{$K$}
  \endrelabelbox
 \caption{The intersection of $K'$ with $S$ for the finger move depicted here is positive.  The sign of the intersection switches if the central ``$\bullet$'' is replaced by ``$\times$'' or if the transverse orientation of $S$ is reversed. }\label{fig:intersect}
 \end{center}
 \end{figure}

% In order to compute the rational Thurston Bennequin number, we count the number of times an endpoint of a curve passes across the dot indicating $K'$ in the sequence of $C_t$ which extends $S$ across the solid torus.

\begin{proof}[Proof of Equation~(\ref{eq:tb})]
  As discussed above, it suffices to examine how the generalized Seifert algorithm extends the Seifert surface $S$ across a fibered neighborhood of a double point of $\Gamma$. The only interactions of $S$ and $K'$ will be when the generalized Seifert algorithm uses  finger moves to push $S$ across the bottom of $N$.  The sign of these intersections may be computed combinatorially as in Figure~\ref{fig:intersect}.  We need to count (with sign) finger moves of $S$ across the bottom of $\partial N$.

The first step in extending $S$ requires sliding each intersection between the fiber and the top edge of $C_t$ into $A_{II}$ and then standardizing the resulting pattern.  Isotope the intersections from $A_{III}$ and $A_{IV}$ to the left across discs where $K$ is oriented to point into the page, and isotope the intersections from $A_I$ to the right across a disc where $K$ is oriented to point out of the page.  Figure~\ref{fig:standardize} shows that moving all the intersections and standardizing the resulting pattern contributes
\[ 2f_{IV}+f_{III}+f_I \]
to the signed intersection number.

\begin{figure}[ht]
  \relabelbox \footnotesize{
     \centerline{\scalebox{.6}{\includegraphics{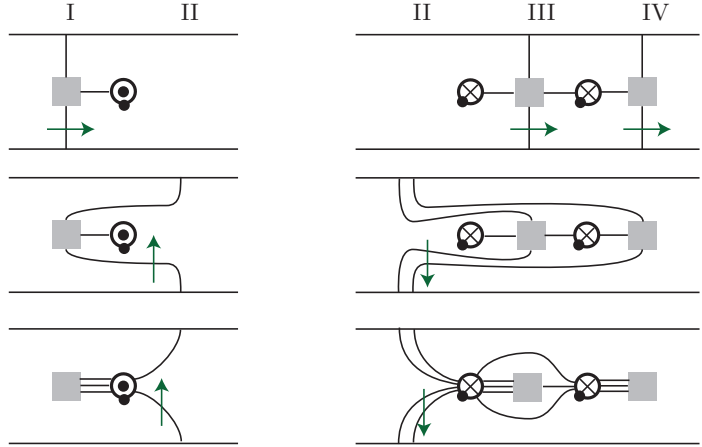}}}}
   \relabel{0}{I}
   \relabel{1}{II}
   \relabel{4}{II}
   \relabel{5}{III}
   \relabel{6}{IV}
   \endrelabelbox
   \caption{Isotopy and standardization: The first step slides the intersections of $A_*$ into $A_{II}$, while the second standardizes the diagram via finger moves.  In this case, $f_*>0$, so the contribution from regions I and III is $+1$, while the contribution from region IV is $+2$.  If the sign of $f_*$ changes, so does the sign of the contribution.}
   \label{fig:standardize}
\end{figure}

Performing saddle moves to eliminate all the longitudinal curves in the pattern does not change the intersection number.  Furthermore, observe that the weight of the curves intersecting each side of the $N$ disc is preserved by the standardization process.  

\begin{figure}[ht]
 \begin{center}
 \scalebox{.6}{\includegraphics{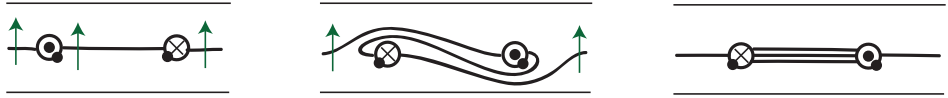}}
 \caption{ Standardizing after the strands of $K$ cross introduces $-m_{II}+m_{III}$ intersections between $K'$ and $S$.
  }\label{fig:steps}
 \end{center}
 \end{figure}

When the strands of $K$ cross, the two $N$ discs on the edges of $A_{II}$ exchange places.  Standardizing the resulting pattern introduces an additional $-m_{II}+m_{III}=-r$ intersections between $S$ and $K'$.  Summing these with the previous intersections and repeating the process at every solid torus yields the following formula for the rational Thurston Bennequin number:
\[ tb_{\mathbb{Q}}(K)=\frac{1}{r}\sum_i (-r+f^i_{I}+f^i_{III}+2f^i_{IV}).\]

To make the formula more elegant, we repeat the same computation, but this time isotope all the intersections to $A_{IV}$ instead.  Counting intersections yields:
\[ tb_{\mathbb{Q}}(K)=\frac{1}{r}\sum_i (-r-f^i_{III}-f^i_{I}-2f^i_{II}).\]

We sum the two formulae for $tb_{\mathbb{Q}}(K)$ and divide by $2$, which yields the desired formula:

\[ \text{tb}_{\mathbb{Q}}(K)=\frac{1}{r}\sum_i (-r-f^i_{II}+f^i_{IV}).\]
\end{proof}

\bibliographystyle{amsplain} 
\bibliography{main}

\end{document}